\documentclass[11pt]{article}
\usepackage{amssymb,latexsym,amsmath,epsfig,amsthm} 

\makeatletter

\renewcommand\section{\@startsection {section}{1}{\z@}
{-30pt \@plus -1ex \@minus -.2ex}
{2.3ex \@plus.2ex}
{\normalfont\normalsize\bfseries\boldmath}}

\renewcommand\subsection{\@startsection{subsection}{2}{\z@}
{-3.25ex\@plus -1ex \@minus -.2ex}
{1.5ex \@plus .2ex}
{\normalfont\normalsize\bfseries\boldmath}}

\renewcommand{\@seccntformat}[1]{\csname the#1\endcsname. }

\makeatother

\newtheorem{theorem}{Theorem}
\newtheorem{lemma}{Lemma}

\newtheorem{corollary}{Corollary}

\theoremstyle{definition}
\newtheorem{definition}{Definition}

\newtheorem{remark}{Remark}

\usepackage{tikz,fullpage}
\usetikzlibrary{arrows,%
                petri,%
                topaths}%
\usepackage{tkz-berge}
\def\ap#1{{\left<{#1}\right>}}

\begin{document}




\title{ON C-POLYNOMIAL FACTORISATIONS}
\author{Harry Hylock, Matthew C. Lettington, Karl Michael Schmidt\\Cardiff School of Mathematics }
\date{\today}

\maketitle

\begin{abstract}
{\it c\/}-Polynomials are polynomials whose (non-zero) coefficients are all
equal to 1.
Taking a theorem by Carlitz and Moser as a motivation, we prove a structure
theorem for the factorisations of the {\it c\/}-polynomial $(x^n-1)/(x-1)$ into
{\it c\/}-irreducible factors by relating factorisations into $c$-polynomials
to joint ordered factorisations arising from the integer $n$.
The question of how many different joint ordered factorisations there are leads
to the precise chromatic polynomial for path graphs, counting the ways of
path graph colourings with $m$ colours under the condition that all colours
are used. Moreover, we give similar counting formulae for joint ordered
factorisations under the constraints that all factors are primes, or that all
factors are square-free.
\end{abstract}

\section{Introduction}
\begin{definition}
A {\it c-polynomial\/} is a polynomial with all non-zero coefficients equal to
1, i.e.\ of the general form
$$
 C_A(x) := \sum_{a \in A} x^a = x^{a_1} + x^{a_2} + \cdots + x^{a_k},
$$
with a finite set of exponents
$A = \{a_1, a_2, \dots, a_k\} \subset \mathbb{N}_0$.
\end{definition}
In this paper we will frequently refer to the {\it c\/}-polynomials that
include consecutive powers starting from 0, the polynomials
$$C_\ap{n} = \sum_{j=0}^{n-1} x^j = \frac{x^n - 1}{x - 1},$$
which have as roots all $n$-th roots of unity except 1.
For the arithmetic progression we here use the notation
$\ap{n} := \{0, 1, 2, \cdots, n-1\}$.
We call the polynomial $C_\ap{1} = C_{\{0\}} = 1$ trivial.
\begin{definition}
A {\it c-factorisation\/} is the polynomial factorisation of a {\it c\/}-polynomial into non-trivial {\it c\/}-polynomial factors.
A {\it c\/}-polynomial is called {\it c-irreducible\/} if there are no {\it c\/}-factorisations of it with more than one factor.
\end{definition}
\begin{remark}
We stress the difference between the subset of {\it c\/}-polynomials in the polynomial ring $\mathbb{Z}[x]$
on the one hand and the polynomial ring $\mathbb{Z}_2[x]$ over the field
$\mathbb{Z}_2 = \mathbb{Z}/(2\mathbb{Z})$ on the other.
While both sets formally contain the same polynomials (with non-zero coefficients
1 only),
the algebraic properties are very different.
For example, the {\it c\/}-polynomial
$1 + x^2 = C_2(x^2)$ is {\it c\/}-irreducible (see Lemma \ref{lem:cirr} below)
whereas in $\mathbb{Z}_2[x]$ the polynomial
$1 + x^2 = (1 + x)^2$ is evidently reducible.
\end{remark}
We'll also need the concept of joint ordered factorisation of an ordered set
of integers, as introduced in \cite{2017SASS}.
The underlying idea is to split given numbers $n_1, n_2, \dots, n_m$ into
non-trivial factors (i.e.\ factors greater than 1) and then put all these
factors into an ordered list such that no two consecutive factors arise from
the same number $n_j$.
\begin{definition}
Let $m\in\mathbb{N}$ and $n_1, \dots, n_m \in \mathbb{N}$, $n_j > 1$ $(j\in\{1, \dots, m\})$.
Then, with $L \in \mathbb{N}$, we call the ordered set of pairs
$$((j_1, f_1), (j_2, f_2), \dots, (j_L, f_L))$$
a {\it joint ordered factorisation of\/} $(n_1, n_2, \dots, n_m)$ if
$j_1, j_2, \dots, j_L \in \{1, \dots, m\}$ with $j_l \neq j_{l-1}$ $(l \in \{2, \dots, L\})$ and
$f_1, f_2, \dots, f_L \in \mathbb{N}$, $f_l > 1$ $(l \in\{1, \dots, L\})$ are such
that
$$\prod_{l : j_l = j} f_l = n_j,\qquad(j \in \{1, \dots, m\}).$$
We refer to $j_l$ as indices and to $f_l$ as factors in the joint ordered
factorisation.
\end{definition}
Joint ordered factorisations are used in the present work as a tool to identify
the structure of {\it c\/}-factorisations and to count them.
This often involves considering not just the joint ordered factorisations of
a fixed $m$-tuple $(n_1, n_2, \dots, n_m)$ but of all such tuples
with given product
$n = \prod_{j=1}^m n_j$, with fixed or variable $m$.
We refer to these groups of joint ordered factorisations
as follows.
\begin{definition}
Let $n\in\mathbb{N}$, $n > 1$.
Let $m\in\mathbb{N}$. Then every joint ordered factorisation of a tuple
$(n_1, n_2, \dots, n_m)$ such that $n = \prod_{j=1}^m n_j$ is called an
{\it $m$-part joint ordered factorisation arising from\/} $n$.

More generally, a {\it joint ordered factorisation arising from\/} $n$ is
an $m$-part joint ordered factorisation arising from $n$ for any suitable $m$.
\end{definition}
For $n\in\mathbb{N}$, $\Omega(n)$ denotes the number of all (not necessarily distinct)
prime factors of $n$ and $\omega(n)$ denotes the number of different prime
factors of $n$.
Thus $n = p_1^{\alpha_1}\,p_2^{\alpha_2} \cdots p_{\omega(n)}^{\alpha_{\omega(n)}}$ with $p_1, p_2, \dots, p_{\omega(n)}$ distinct primes and
$\alpha_1, \alpha_2, \dots, \alpha_{\omega(n)} \in \mathbb{N}$; then
$\Omega(n) = \sum_{j=1}^{\omega(n)} \alpha_j$.

\section{{\it c\/}-irreducible {\it c\/}-factorisations}
Carlitz and Moser have the following statement \cite[Theorem 1]{C&M66}.
\begin{theorem}
\label{thm:CM}
Let $n = \prod_{j=1}^{\Omega(n)} p_j$ with (not necessarily distinct) primes
$p_j$ $(j \in \{1, \dots, \Omega(n)\})$.
Then
$$
 C_\ap{n}(x) = C_\ap{p_1}(x)\,C_\ap{p_2}(x^{p_1})\,C_\ap{p_3}(x^{p_1 p_2}) \cdots C_\ap{p_{\Omega(n)}}(x^{p_1 p_2 \cdots p_{\Omega(n)-1}}),
$$
where each factor on the right-hand side is {\it c}-irreducible.
Moreover, all factorisations of $C_\ap{n}$ into {\it c\/}-irreducible factors are obtained in this way.
\end{theorem}
In the following, we give a proof of this theorem, noting that the proof
provided in \cite{C&M66} is incomplete, see Remark \ref{rmk:CM} below.
\begin{lemma}
\label{lem:hls}
Let $n \in\mathbb{N}, n > 1$.
Every ordered $c$-factorisation of $C_\ap{n}$ into $m$ $c$-polynomials corresponds to a
unique $m$-part joint ordered factorisation arising from $n$ and vice versa.

Specifically, the $j$-th {\it c}-polynomial factor in the {\it c}-factorisation has the form
$\prod_{l : j_l = j} C_\ap{f_l}(x^{F_l})$,
where
$((j_1, f_1), (j_2, f_2), \dots, (j_L, f_L))$ is the joint ordered factorisation
and
$F_l := \prod_{s=1}^{l-1} f_s$ $(l \in \{1, \dots, L\})$.
\end{lemma}
\begin{proof}
Suppose
$$
 C_\ap{n} = C_{A_1} C_{A_2} \cdots C_{A_m}
$$
with finite non-empty sets $A_j \subset \mathbb{N}_0$, $A_j \neq \{0\}$ $(j \in \{1, \dots, m\})$.
As
\begin{align*}
 \prod_{j=1}^m C_{A_j}(x) &= \prod_{j=1}^m \sum_{\alpha_j \in A_j} x^{\alpha_j}
 = \sum_{(\alpha_1, \dots, \alpha_m) \in A_1 \times \cdots \times A_m} x^{\alpha_1 + \alpha_2 + \cdots + \alpha_m}
 = \sum_{\alpha \in \sum_{j=1}^n A_j} \mu_{\alpha} x^\alpha,
\end{align*}
where $\mu_\alpha$ is the multiplicity of $\alpha$ in forming the set sum
$\sum_{j=1}^n A_j$,
it is evident that $C_\ap{n} = \prod_{j=1}^m C_{A_j}$
if and only if
$\sum_{j=1}^m A_j = \ap{n}$ with $\mu_\alpha = 1$ for all $\alpha$, i.e.\ if and only if
$A_1, A_2, \dots, A_m$ form an $m$-part sum system of
cardinalities
$|A_j| = n_j$ $(j \in\{1, \dots, m\})$ such that $\prod_{j=1}^m n_j = n$.
By \cite[Theorem 6.7]{2017SASS}, there is a unique joint ordered
factorisation
$((j_1, f_1), (j_2, f_2), \dots, (j_L, f_L))$
of
$(n_1, n_2, \dots, n_m)$ such that
\begin{equation}
\label{eq:sss}
 A_j = \sum_{l : j_l = j} F_l\, \ap{f_l}
 \qquad (j \in\{1, \dots, m\}),
\end{equation}
where
$F_l := \prod_{s=1}^{l-1} f_s$,
and conversely every such joint ordered factorisation generates a sum system.
The corresponding polynomial factors are
\begin{align*}
 C_{A_j} &= \sum_{a \in A_j} x^a
 = \prod_{l : j_l = j} \sum_{a_l \in\ap{f_l}} x^{F_l\,a_l}
 = \prod_{l : j_l = j} \sum_{a_l \in\ap{f_l}} \left(x^{F_l}\right)^{a_l}
 = \prod_{l : j_l = j} C_\ap{f_l}(x^{F_l}).
\end{align*}
\end{proof}

The ``if'' part of the following statement was shown in \cite[Lemma 3]{C&M66}; we include a
proof for the reader's convenience.
\begin{lemma}
\label{lem:cirr}
Let $f, a$ be non-negative integers, $f > 1$. Then the polynomial
$C_\ap{f}(x^a)$ is {\it c}-irreducible if and only if $f$ is prime.
\end{lemma}
\begin{proof}
If $f$ is not prime, so $f = f_1 f_2$ with integers $f_1, f_2 > 1$, then
$$
 C_\ap{f}(x^a) = \frac{x^{a f} - 1}{x^a - 1}
 = \frac{x^{a f_1 f_2} - 1}{x^{a f_1} - 1} \, \frac{x^{a f_1} - 1}{x^a - 1}
 = C_\ap{f_2}(x^{a f_1}) \, C_\ap{f_1}(x^a)
$$
is {\it c}-reducible.

Further, by Lemma \ref{lem:hls}
every {\it c}-factorisation of $C_\ap{f}$ corresponds to a joint ordered
factorisation arising from $f$; if $f$ is prime, this can only be
$((1, f))$, having a single factor. Therefore $C_\ap{f}$ is {\it c}-irreducible.
Finally, suppose
$C_\ap{f}(x^a) = C_A(x)\,C_B(x)$ with finite $A, B \subset \mathbb{N}_0$, for $f$ prime. Then $0 \in A$ and $0 \in B$ to produce the constant term $1$ in $C_\ap{f}(x^a)$.
If $k \in A$ and $k \notin a \mathbb{N}_0$, then $C_\ap{f}(x^a)$ will have the
term $x^k\,x^0 = x^k$, which cannot be cancelled by any other partial product
$x^\alpha\, x^\beta$ with $\alpha\in A$ and $\beta\in B$, as all coefficients
are non-negative. This gives a contradiction, so there is some $\tilde A \subset
\mathbb{N}_0$ such that $A = a \tilde A$. By analogous reasoning, there is some
$\tilde B \subset \mathbb{N}_0$ such that $B = a \tilde B$.
But then
 $C_\ap{f}(x^a) = C_{\tilde A}(x^a)\,C_{\tilde B}(x^a)$
and consequently $C_\ap{f} = C_{\tilde A}\,C_{\tilde B}$.
As $C_\ap{f}$ is {\it c}-irreducible, it follows that $\tilde A = \{0\}$ or
$\tilde B = \{0\}$ and hence $A = \{0\}$ or $B = \{0\}$, so
$C_\ap{f}(x^a)$ is also {\it c}-irreducible.
\end{proof}
\begin{proof}[Proof of Theorem \ref{thm:CM}]
The first statement follows by repeated application of the expansion used
in the first part of the proof of Lemma \ref{lem:cirr}
\begin{align*}
 C_\ap{p_1 p_2 \cdots p_{\Omega(n)}} &= \frac{x^{p_1 p_2 p_3 \cdots p_{\Omega(n)}} - 1}{x - 1}
 = \frac{x^{p_1} - 1}{x - 1}\,\frac{(x^{p_1})^{p_2 p_3 \cdots p_{\Omega(n)}} - 1}{x^{p_1} - 1}
\\
 &= \frac{x^{p_1} - 1}{x - 1}\,\frac{(x^{p_1})^{p_2} - 1}{x^{p_1} - 1}\,\frac{(x^{p_1 p_2})^{p_3 \cdots p_{\Omega(n)}} - 1}{x^{p_1 p_2} - 1}
\\
 &= \dots
 = \frac{x^{p_1} - 1}{x - 1}\,\frac{(x^{p_1})^{p_2} - 1}{x^{p_1} - 1}\,\frac{(x^{p_1 p_2})^{p_3} - 1}{x^{p_1 p_2} - 1} \cdots \frac{(x^{p_1 p_2 \cdots p_{\Omega(n) - 1}})^{p_{\Omega(n)}} - 1}{x^{p_1 p_2 \cdots p_{\Omega(n) - 1}} - 1}
\\
 &= C_\ap{p_1}(x)\,C_\ap{p_2}(x^{p_1})\,C_\ap{p_3}(x^{p_1 p_2}) \cdots C_\ap{p_{\Omega(n)}}(x^{p_1 p_2 \cdots p_{\Omega(n)-1}})
\end{align*}
and Lemma \ref{lem:cirr}.

For the second statement, consider a {\it c}-factorisation of $C_\ap{n}$.
By Lemma \ref{lem:hls}, there is a corresponding joint ordered factorisation
$((j_1, f_1), (j_2, f_2), \dots, (j_L, f_L))$ arising from $n$ and each of
the {\it c}-polynomial factors has the form
$\prod_{l : j_l = j} C_\ap{f_l}(x^{F_l})$, for some $j$. If the product has
more than one term, this polynomial is obviously {\it c}-reducible.
In conjunction with Lemma \ref{lem:cirr}, this shows that the {\it c}-polynomial factor is {\it c}-irreducible (if and) only if there is only one $l \in \{1, \dots, L\}$ such that $j_l = j$, and $f_l$ is prime.
The factor then takes the form $C_\ap{f_l}(x^{F_l})$.

For a factorisation of $C_\ap{n}$ into {\it c}-irreducible factors, the joint
ordered factorisation must have all $L~=~\Omega(n)$ indices distinct and all
factors $f_1, \dots, f_L$ prime.
Up to permutation of the primes, this joint ordered factorisation has the
form
$((1, p_1), (2, p_2), \dots, (\Omega(n), p_{\Omega(n)}))$
and gives rise
to the factorisation stated in the theorem.
\end{proof}
\begin{remark}
\label{rmk:CM}
In \cite{C&M66} the proof of the second statement of Theorem \ref{thm:CM}
relies on the following statement
\cite[Lemma 1]{C&M66}.

{\it
``Let
$$\frac{x^n - 1}{x - 1} = A(x) B(x)$$
where $A(x)$ and $B(x)$ are {\it c}-polynomials. Then either $A(x)$ or $B(x)$
is of the form ${(x^r - 1)/(x - 1)}$ where $r$ is a divisor of $n$.''
\/}

However, this statement is incorrect; in fact the connection between
{\it c}-factorisations and joint ordered factorisations set out in Lemma
\ref{lem:hls} above easily yields counterexamples. For example, the
joint ordered factorisation $((1,2),(2,3),(1,2))$ gives
$$
 \frac{x^{12}-1}{x-1} = C_\ap{12}(x) = (1 + x + x^6 + x^7) (1 + x^2 + x^4)
$$
with neither factor of the form claimed in the above statement; the joint
ordered factorisation $$((1, 2),(2, 3),(1, 2),(2, 2))$$ gives
$$
 \frac{x^{24}-1}{x-1} = C_\ap{24}(x) = (1 + x + x^6 + x^7) (1 + x^2 + x^4 + x^{12} + x^{14} + x^{16})
$$
with neither factor even of the form $(x^{ra}-1)/(x^a-1)$.
\end{remark}
As already observed in \cite{C&M66} and apparent from both the {\it\ c\/}-factorisation stated in Theorem \ref{thm:CM} and the corresponding joint ordered
factorisation arising from $n$ with all factors prime, the number of all
different factorisations of $C_\ap{n}$ into (necessarily $\Omega(n)$)
{\it c\/}-irreducible factors is equal, up to permutation of factors, to the
number of linear arrangements of all prime factors of $n$.
Thus, writing $n = p_1^{\alpha_1} p_2^{\alpha_2} \cdots p_{\omega(n)}^{\alpha_{\omega(n)}}$, where we take $p_1, \dots, p_{\omega(n)}$ to be distinct
prime factors and $p_{\omega(n)+1}, \dots, p_{\Omega(n)}$ to be repetitions of
these, we obtain the following.
\begin{corollary}
\label{cor:cirr}
Let $n = p_1^{\alpha_1} p_2^{\alpha_2} \cdots p_{\omega(n)}^{\alpha_{\omega(n)}}$ with distinct primes $p_1$, $p_2$, \dots, $p_{\omega(n)}$ and $\alpha_j \in \mathbb{N}$ $(j \in \{1, \dots, \omega(n)\})$. Then, up to permutation of factors,
$C_\ap{n}$ has
$$
 \begin{pmatrix} \Omega(n) \\ \alpha_1\ \alpha_2\ \dots\ \alpha_{\omega(n)} \end{pmatrix}
 = \frac{(\alpha_1 + \alpha_2 + \cdots + \alpha_{\omega(n)})!}{\alpha_1!\,\alpha_2!\,\cdots\,\alpha_{\omega(n)}!}
$$
different {\it c\/}-irreducible {\it c\/}-factorisations.
\end{corollary}
\section{{\it c\/}-factorisations and precise chromatic polynomial}
In Corollary \ref{cor:cirr} we gave the count of {\it c\/}-irreducible
{\it c\/}-factorisations of $C_\ap{n}$.
We now consider the question of how many different {\it c\/}-factorisations
there are without the constraint of having {\it c\/}-irreducible factors.
As noted in the beginning of the proof of Lemma \ref{lem:hls}, there is a
one-to-one relationship between ordered {\it c\/}-factorisations of $C_\ap{n}$ into
$m$ factors and $m$-part sum systems; by \cite[Theorem 6.7]{2017SASS} the latter
are in one-to-one relationship with the $m$-part joint ordered factorisations
arising from $n$.
Note that we are here considering ordered factorisations of $C_\ap{n}$,
corresponding to keeping the order of the component sets of the sum system
fixed. If we wish to count {\it c\/}-factorisations irrespective of the order
of factors, then we need to divide the count by $m!$, with $m$ the number of
factors; note that, due to their structure stated in Equation (\ref{eq:sss}),
no two component sets of a sum system are equal
and hence no two {\it c\/}-polynomial factors in a {\it c\/}-factorisation of
$C_\ap{n}$ are the same.

Let us at first consider fixed $m$. Then the counting of ordered
{\it c\/}-factorisations of $C_\ap{n}$ with $m$ factors is equivalent to counting all
$m$-part joint ordered factorisations arising from $n$.
In view of the definition of joint ordered factorisations, each of these
can be constructed by taking an ordered factorisation of $n$ into $L$ non-trivial factors
$f_1, f_2, \dots, f_L$
and then assigning them indices $j_1, j_2, \dots, j_L \in \{1, \dots, m\}$
under the constraints that $j_l \neq j_{l-1}$ for all $j \in \{2, \dots, L\}$
and that each value in $\{1, \dots, m\}$ is taken by at least one index.
These two processes are independent, apart from being linked by the
variable $L$, and can therefore be considered as two separate counting problems.
Note that only values $L \in \{m, \dots, \Omega(n)\}$ give rise to valid
joint ordered factorisations.

The first of these counting problems is solved by the non-trivial divisor
function $c_L$ that counts the number of different ordered factorisations of
an integer into $L$ factors greater than 1.
This arithmetic function already studied by MacMahon \cite{MacMah} can be expressed in the form
$$
c_L(n) = \sum_{k=0}^L(-1)^k\binom{L}{k}\prod_{j=1}^{\omega(n)}\binom{\alpha_j+L-k-1}{\alpha_j}
$$
where $n=p_1^{\alpha_1}\, p_2^{\alpha_2}\cdots p_{\omega(n)}^{\alpha_{\omega(n)}}$ with distinct primes $p_1, p_2, \dots, p_{\omega(n)}$
(see \cite[Theorem 1(b)]{LSDF&NSS}, where also its definition in terms of
Dirichlet convolution and generalisations are discussed).

The second counting problem of how many ways there are of giving a row of
$L$ objects labels of $m$ different types such that each type occurs at least
once and no adjacent objects have a label of the same type is clearly
equivalent to the counting of colourings of a path graph of $L$ vertices
with $m$ colours such that each colour is used at least once, under the
standard condition that connected vertices have different colours.
If the first condition of using every colour is omitted, this is the usual
graph colouring problem for a path graph, its answer being given by the
{\it chromatic polynomial\/}
$$
 P(L, m) = m\,(m-1)^{L-1};
$$
indeed, going through the vertices in turn from one end of the path graph
to the other, the first vertex can have any colour and each following
vertex can have any colour except that of the preceding vertex.
We introduce the additional constraint in the following concept.
\begin{definition}
Let $L, m \in \mathbb{N}$. The {\it precise chromatic polynomial\/}
$P^\#(L, m)$ counts the number of graph colourings of the path graph with
$L$ vertices using $m$ colours such that every colour is used at least once.
\end{definition}
The precise chromatic polynomial can be expressed in terms of the Stirling
numbers of the second kind,
$$
{L \brace m} := S(L, m) = \frac 1{m!} \sum_{k=0}^m (-1)^{m-k} {m \choose k} k^L
$$
(see \cite[Theorem 5.1.A]{Comtet}) as follows.
\begin{theorem}
\label{thm:pcp}
Let $L, m \in \mathbb{N}$. Then
\begin{equation}
\label{IDIDIT}
 P^\#(L, m) = m!\,{L-1 \brace m-1}.
\end{equation}
\end{theorem}
\begin{remark}
\label{rmk:pcp}
(a)
From its definition, it is clear that $P^\#(L, m) = 0$ if $L < m$ as it is
impossible to use all $m$ colours on only $L$ vertices; this case is covered
in Equation (\ref{IDIDIT}) because
the Stirling number then vanishes.

(b)
The precise chromatic polynomial counts all path graph colourings using
exactly $m$ fixed colours. If the colours are treated as interchangeable,
so any two colourings that differ only be a permutation of colours are
identified, then by Equation (\ref{IDIDIT}), the number of (equivalence
classes of) colourings is ${L-1 \brace m-1}$.

(c)
The above outcome agrees with Munagi's result for a combinatorially equivalent
problem \cite[Remark 2.2]{Munagi} that was further extended in \cite[Section 3]{Duncan2009}.
\end{remark}
\begin{proof}[Proof of Theorem \ref{thm:pcp}]
We first note that
by Remark \ref{rmk:pcp}(a),
$P^\#(1, m) = \delta_{m, 1}$ in accordance with Equation (\ref{IDIDIT}).
Hence we assume $L > 1$ without loss of generality in the following.

Since general path graph colourings with (up to) $m$ colours can be classified
into path graph colourings where exactly $n$ colours are used, for $n\in\{1, \dots, m\}$, the chromatic polynomial and the precise chromatic polynomial are
related as
$$
P(L, m) = \sum_{n=1}^m P^\#(L, n){m \choose n}.
$$
This gives the recurrence relation for the precise chromatic polynomial
\begin{equation}
\label{eq:pcprecur}
P^\#(L, m) = m\,(m-1)^{L-1} - \sum_{n=1}^{m-1} P^\#(L, n) {m \choose n}
\qquad (m \in \mathbb{N}).
\end{equation}
Note that this relation already produces the initial value
$P^\#(L, 1) = 0$
which is correct as no path graph with more than one vertex has a valid
colouring with only one colour.

On the other hand, there is the following identity involving Stirling numbers
of the second kind and falling factorials
$x^{\underline{n}} = \prod_{j=0}^{n-1} (x - j)$,
$$
x^k = \sum_{n=0}^k {k \brace n}\,x^{\underline{n}}
\qquad (k \in \mathbb{N}_0)
$$
(see \cite[Equation (6.10)]{GKP}).
Hence
$$
 (m-1)^{L-1} = \sum_{n=0}^{L-1} {L-1 \brace n} (m-1)^{\underline{n}}
 = \sum_{n=1}^L {L-1 \brace n-1} (m-1)^{\underline{n-1}}
 = \sum_{n=1}^L {L-1 \brace n-1} \frac{(m-1)!}{(m-n)!}
$$
and thus
$$
 m\,(m-1)^{L-1} = \sum_{n=1}^L {L-1 \brace n-1} \frac{m!}{(m-n)!}
 = \sum_{n=1}^L n! {L-1 \brace n-1} {m \choose n}
 = \sum_{n=1}^m n! {L-1 \brace n-1} {m \choose n},
$$
using in the last step that
${L-1 \brace n-1} = 0$ if $n > L$ and ${m \choose n} = 0$
if $n > m$.
Hence we obtain
\begin{equation}
\label{eq:stirecur}
m! {L-1 \brace m-1} = m\,(m-1)^{L-1} - \sum_{n=1}^{m-1} n! {L-1 \brace n-1}{m \choose n}
\qquad (m \in \mathbb{N}).
\end{equation}
Comparing Equations (\ref{eq:pcprecur}) and (\ref{eq:stirecur}), we see that
$P^\#(L, m)$ and $m! {L-1 \brace m-1}$ satisfy the same recurrence relation,
and since for $m=1$ we have $1! {L-1 \brace 0} = 0 = P^\#(L, 1)$,
Equation (\ref{IDIDIT}) follows.
\end{proof}
Collecting the results in this section, we obtain the following counting
theorem.
\begin{theorem}
\label{thm:cfcnt}
Let $m,n \in \mathbb{N}$, $n > 1$. Then the number of different {\it c\/}-factorisations of $C_\ap{n}$ into $m$ non-trivial factors, irrespective of the order of factors, is equal to
\begin{equation}
\label{eq:cfcnt}
\tilde{\cal N}_m(n) = \sum_{L=m}^{\Omega(n)} {L-1 \brace m-1} c_L(n)
\end{equation}
and the number of all different {\it c\/}-factorisation of $C_\ap{n}$,
irrespective of the order of factors, is equal to
$$
\tilde{\cal N}(n) = \sum_{m=1}^{\Omega(n)} \sum_{L=m}^{\Omega(n)} {L-1 \brace m-1} c_L(n).
$$
\end{theorem}
\begin{remark}
(a)
The above considerations also give a formula for the number of all $m$-part
joint ordered factorisations arising from $n$,
$$
{\cal N}_m(n) = \sum_{L=m}^{\Omega(n)} m! {L-1 \brace m-1} c_L(n)
$$
(which is also the number of ordered {\it c\/}-factorisations of $C_\ap{n}$
into $m$ non-trivial factors).
Interestingly, for the same counting problem there is also the formula
\begin{equation}
\label{eq:Nm}
{\cal N}_m(n) = \sum_{L=0}^{\Omega(n)} m! {L \brace m} c_L^{(-L)}(n)
\end{equation}
(see \cite[Theorem 1.5]{L2025MF&DF}), where $c_L^{(-L)}$ is an associated
divisor function first introduced in \cite{HHLS} that can be interpreted as
giving, up to a sign $(-1)^{\Omega(n)+L}$, the number of ordered factorisations
into $L$ non-trivial, square-free factors (see \cite{LSDF&NSS}).
These two formulae bear a superficial resemblance but with clear differences.
Nevertheless, 
the two expressions can be reconciled (see \cite[Theorem 4.1]{LSfut}). 

(b)
The formula of Equation (\ref{eq:Nm}) arises by summation, over all
ordered factorisations
$n = n_1\,n_2\cdots n_m$ with factors greater than 1, of the number of
joint ordered factorisations of $(n_1, n_2, \dots, n_m)$, which is given by
\cite[Theorem 4]{LSDF&NSS}
\begin{equation}\label{eq:Ntup}
{\cal N}_{(n_1, n_2, \dots, n_m)} = \sum_{l_1=1}^\infty \sum_{l_2=1}^\infty \cdots \sum_{l_m=1}^\infty {l_1+l_2+\cdots+l_m \choose l_1\ l_2\ \dots\ l_m}
\prod_{j=1}^m c_{l_j}^{(-l_j)}(n_j).
\end{equation}
The proof of this formula involves a more intricate combinatorial problem
that can be stated as counting the path graph colourings applying
colour $j$ on exactly $n_j$ vertices, for all $j \in \{1, \dots, m\}$.

(c)
In Section 3 of \cite{C&M66}, Carlitz and Moser consider the number of
{\it c\/}-factorisations
$C_\ap{n}(x) = A(x)\,B(x)$,
irrespective of order,
and state it as
\begin{equation}\label{eq:CMcnt}
 \sum_{r=0}^{\Omega(n)} c_r(n)
\end{equation}
(using the notation $T_r'$ for what we call $c_r$).
This differs by $c_1(n) = 1$ from Equation (\ref{eq:cfcnt}) for $m=2$, as $c_0(n) = 0$ for
$n > 1$ and ${L-1 \brace 1} = 1$ for all $L \ge 2$.
However, the additional term $c_1(n)$ would correspond to not factorising $n$
at all, which makes one of the factors $A(x)$ or $B(x)$ trivial, as
clearly not intended in \cite{C&M66}.
For example, consider $n = 6$; there are only two {\it c\/}-factorisations
with exactly two factors (not counting permutation of factors)
$$
C_6(x) = (x^2 + x + 1) (x^3 + 1) = (x^4 + x^2 + 1) (x + 1)
$$
in agreement with Equation (\ref{eq:cfcnt}),
$\tilde{\cal N}_2(6) = \sum_{L=2}^2 {1 \brace 1} c_2(6) = c_2(6) = 2$,
not $3$ as suggested by Formula (\ref{eq:CMcnt}).
\end{remark}

\section{Conditioned joint ordered factorisations}
In this section we extend the counting arguments for joint ordered
factorisations arising from an integer $n > 1$ to joint ordered factorisations
with additional constraints.
We consider two types of constraints, factorisation into prime factors,
motivated by Theorem \ref{thm:CM}, and factorisation in square-free factors.
\begin{definition}
We call joint ordered factorisations where all factors are primes
{\it maximal joint ordered factorisations\/}.
\end{definition}
To count the number of $m$-part maximal joint ordered factorisations arising
from $n$, we consider, as in the proof of Theorem \ref{thm:cfcnt}, the two
separate problems of counting all different ordered factorisations of $n$
into primes, and of counting all different ways of assigning indices from
$\{1, \dots, m\}$, taking all $m$ values, to the ordered list of prime factors.
The first problem is answered in Corollary \ref{cor:cirr} above, the second by
the precise chromatic polynomial $P^\#(\Omega(n), m)$ in Theorem \ref{thm:pcp}. Combining these
results, we obtain the following.
\begin{theorem}
Let $m, n\in\mathbb{N}$, $n > 1$.
Then the number of $m$-part maximal joint ordered factorisations arising from
$n$
is
$$
{\cal N}_m^{\rm max}(n) = {\Omega(n) \choose \alpha_1\ \alpha_2\ \dots\ \alpha_{\omega(n)}} m! {\Omega(n) - 1 \brace m-1}
$$
and the number of all maximal joint ordered factorisations arising from $m$ is
$$
{\cal N}^{\rm max}(n) = {\Omega(n) \choose \alpha_1\ \alpha_2\ \dots\ \alpha_{\omega(n)}} \sum_{m=1}^{\Omega(n)} m! {\Omega(n) - 1 \brace m-1},
$$
where $\alpha_1, \dots, \alpha_{\omega(n)}$
in the multinomial coefficient
arise from the prime decomposition of $n = p_1^{\alpha_1}\, \cdots p_{\omega(n)}^{\alpha_{\omega(n)}}$.
\end{theorem}
Noting the relationship between the Bell numbers and the Stirling numbers of the second kind
\cite[Equation 5.4.{[4a]}]{Comtet}
$$
B_j = \sum_{k=0}^j {j \brace k} \qquad (j \in \mathbb{N}_0)
$$
yields the following variant where we identify joint ordered factorisations that differ only by a
permutation of the index values in $\{1, \dots, m\}$.
\begin{corollary}
Let $m, n \in\mathbb{N}$, $n > 1$. Then the number of $m$-part maximal joint
ordered factorisations arising from $n$, counted irrespective of the order of
parts, is
$$
\tilde{\cal N}_m^{\rm max}(n) = {\Omega(n) \choose \alpha_1\ \alpha_2\ \dots\ \alpha_{\omega(n)}} {\Omega(n) - 1 \brace m - 1}
$$
and the number of all maximal joint ordered factorisations arising from $n$,
counted irrespective of the order of parts, is
$$
\tilde{\cal N}^{\rm max}(n) = {\Omega(n) \choose \alpha_1\ \alpha_2\ \dots\ \alpha_{\omega(n)}} B_{\Omega(n) - 1},
$$
where $\alpha_1, \dots, \alpha_{\omega(n)}$
in the multinomial coefficient
arise from the prime decomposition of $n = p_1^{\alpha_1}\, \cdots p_{\omega(n)}^{\alpha_{\omega(n)}}$.
\end{corollary}
\begin{definition}
We call joint ordered factorisations where all factors are square-free
{\it square-free joint ordered factorisations\/}.
\end{definition}
For the counting of the number of $m$-part square-free joint ordered
factorisations arising from $n$, we use the fact that, with the associated divisor
function $c_L^{(-L)}$, which can be calculated via the formula
\cite[Theorem 1(b)]{LSDF&NSS}
$$
c_L^{(-L)}(n) = \sum_{k=0}^L (-1)^k {L\choose k} \prod_{j=1}^{\omega(n)} {\alpha_j - k - 1 \choose \alpha_j},
$$
where $n = p_1^{\alpha_1}\,p_2^{\alpha_2} \cdots p_{\omega(n)}^{\alpha_{\omega(n)}}$ with distinct primes $p_1, \dots, p_{\omega(n)}$,
the number of ordered factorisations of $n$ into square-free non-trivial factors
is given by $|c_L^{(-L)}(n)|$.
The number of different ways of assigning indices from $\{1, \dots, m\}$, not
omitting any value, to the ordered list of factors is again counted by the
precise chromatic polynomial $P^\#(L, m)$ in Theorem \ref{thm:pcp}, and we
obtain the following.
\begin{theorem}\label{thm:sqf}
Let $m, n \in\mathbb{N}$, $n > 1$. Then the number of different $m$-part
square-free joint ordered factorisations arising from $n$ is
$$
{\cal N}_m^{\boxtimes}(n) = \sum_{L=m}^{\Omega(n)} m! {L-1 \brace m-1} \,|c_L^{(-L)}(n)|
$$
and the number of all different square-free joint ordered factorisations
arising from $n$ is
$$
{\cal N}^{\boxtimes}(n) = \sum_{L=1}^{\Omega(n)} \sum_{m=1}^{L} m! {L-1 \brace m-1} \,|c_L^{(-L)}(n)|.
$$
\end{theorem}
Identifying joint ordered factorisations that differ only by a permutation of
the index values in $\{1, \dots, m\}$, we find the following variant.
\begin{corollary}\label{cor:sqf}
Let $m, n \in\mathbb{N}$, $n > 1$. Then the number of $m$-part square-free
joint ordered factorisations arising from $n$, counted irrespective of the
order of parts, is
$$
\tilde{\cal N}_m^\boxtimes(n) = \sum_{L=m}^{\Omega(n)} {L-1 \brace m-1}\,|c_L^{(-L)}(n)|
$$
and the number of all square-free joint ordered factorisations arising from $n$,
counted irrespective of the number of parts, is
$$
\tilde{\cal N}^\boxtimes(n) = \sum_{L=1}^{\Omega(n)} B_{L-1}\,|c_L^{(-L)}(n)|.
$$
\end{corollary}

While the counting functions considered in Theorem \ref{thm:sqf} and
Corollary \ref{cor:sqf} refer to all square-free joint ordered factorisations
arising from $n$, we can also state the number of square-free joint ordered
factorisation of a given $m$-tuple $(n_1, n_2, \dots, n_m)$, in analogy to
the unconstrained count given in Equation (\ref{eq:Ntup}).
Curiously, in the corresponding formula, given in the next theorem, the
non-trivial divisor functions $c_l$ play a role analogous to that of
$c_l^{(-l)}$ in Equation (\ref{eq:Ntup}).
\begin{theorem}
Let $m\in\mathbb N$, $n_1, n_2, \dots, n_m \in\mathbb N$ all be greater than $1$.
Then the number of square-free joint ordered factorisations of $(n_1, n_2, \dots, n_m)$
is
\begin{equation*}
\mathcal N_{(n_1, n_2, \dots, n_m)}^\boxtimes = (-1)^{\Omega(n_1 n_2 \cdots n_m)} \sum_{l_1=1}^\infty \sum_{l_2=1}^\infty \cdots \sum_{l_m=1}^\infty (-1)^{\sum_{j=1}^m l_j} {\sum_{j=1}^m l_j \choose l_1\ l_2\ \dots\ l_m} \prod_{j=1}^m c_{l_j}(n_j).
\end{equation*}
\end{theorem}
\begin{proof}
The formation of a joint ordered factorisation can be considered as taking the
following steps, (1) factorising the number $n_j$ into $\nu_j$ non-trivial
(and, in the present case, square-free) factors, for each $j\in\{1, \dots, m\}$;
(2) labelling the $L = \sum_{j=1}^m \nu_j$ index-factor pairs with indices
$j_1, j_2, \dots, j_L \in \{1, \dots, m\}$ such that
index $j$ appears exactly $\nu_j$ times and
 $j_l \neq j_{l-1}$ for all
$l \in \{2, \dots, L\}$ and (3)
distributing the $\nu_j$ factors of $n_j$ to the pairs with index $j_j$, for
each $j \in \{1, \dots, m\}$.
The different ways of performing step (2) were counted in \cite[Theorem 3]{LSDF&NSS}, giving
\begin{equation*}
e_{(\nu_1, \nu_2, \dots, \nu_m)} = \sum_{k_1=0}^{\nu_1-1} \sum_{k_2=0}^{\nu_2-1} \cdots \sum_{k_m=0}^{\nu_m-1} (-1)^{\sum_{j=1}^m k_j} {\sum_{j=1}^m \nu_j - k_j \choose (\nu_1-k_1)\ (\nu_2-k_2)\ \dots\ (\nu_m-k_m)} \prod_{j=1}^m {\nu_j-1 \choose k_j}.
\end{equation*}
Hence, bearing in mind that $|c_\nu^{(-\nu)}(n)| = (-1)^{\Omega(n)+\nu} c_\nu^{(-\nu)}(n)$ is the number of ordered factorisations of $n$ into $\nu$ non-trivial
square-free factors, we find
\begin{align*}
\mathcal N_{(n_1, \dots, n_m)}^\boxtimes
&= \sum_{\nu_1=1}^\infty \cdots \sum_{\nu_m=1}^\infty e_{(\nu_1, \nu_2, \dots, \nu_m)} \prod_{j=1}^m (-1)^{\Omega(n_j)+\nu_j} c_{\nu_j}^{(-\nu_j)}(n_j)
\\
&= \sum_{\nu_1=1}^\infty \cdots \sum_{\nu_m=1}^\infty \sum_{k_1=0}^{\nu_1-1} \cdots \sum_{k_m=0}^{\nu_m-1}
{\sum_{j=1}^m (\nu_j - k_j) \choose (\nu_1 - k_1)\ \dots\ (\nu_m - k_m) }
\\
&\qquad\qquad\qquad\qquad \times
\prod_{j=1}^m (-1)^{k_j+\Omega(n_j)+\nu_j} {\nu_j-1 \choose k_j} c_{\nu_j}^{(-\nu_j)}(n_j)
\\
&= (-1)^{\Omega(n_1\cdots n_m)} \sum_{l_1=1}^\infty \cdots \sum_{l_m=1}^\infty
(-1)^{\sum_{j=1}^m l_m} {\sum_{j=1}^m l_m \choose l_1\ \dots\ l_m}
\\
&\qquad\qquad\qquad\qquad \times
\sum_{k_1=0}^\infty\cdots\sum_{k_m=0}^\infty \prod_{j=1}^m {k_j+l_j-1 \choose k_j} c_{k_j+l_j}^{(-k_j-l_j)}(n_j)
\\
&= (-1)^{\Omega(n_1\cdots n_m)} \sum_{l_1=1}^\infty \cdots \sum_{l_m=1}^\infty
(-1)^{\sum_{j=1}^m l_m} {\sum_{j=1}^m l_m \choose l_1\ \dots\ l_m}
\\
&\qquad\qquad\qquad\qquad \times
\prod_{j=1}^m \sum_{k=0}^\infty {k+l_j-1 \choose k} c_{k+l_j}^{(-k-l_j)}(n_j).
\end{align*}
Now we use the facts that $c_k$ and $c_k^{(-k)}$ can be defined as Dirichlet
convolution powers,
\begin{equation*}
c_k = (1 - e)^{*k}, \qquad c_k^{(-k)} = (e - \mu)^{*k},
\end{equation*}
where $\mu$ is the M\"obius function, $1$ is the constant function $1$ and
$e$, the arithmetic function such that $e(1) = 1$ and $e(n) = 0$ $(n\in\mathbb N, n \ge 2)$, is the neutral element for the Dirichlet convolution product
in the commutative algebra of arithmetic functions (see \cite{LSDF&NSS} for details), to find
\begin{align*}
\sum_{k=0}^\infty {k+l_j-1 \choose k} c_{k+l_j}^{(-k-l_j)}
&= \sum_{k=0}^\infty {k+l_j-1 \choose k} (e - \mu)^{*(k+l_j)}
\\
&= (e - \mu)^{*l_j} * \sum_{k=0}^\infty (-1)^k {k+l_j-1 \choose k} (\mu - e)^{*k}
\\
&= (e - \mu)^{*l_j} * (e + (\mu - e))^{*(-l_j)}
= (e - \mu)^{*l_j} * \mu^{*(-l_j)}
\\
&= (1 - e)^{*l_j} = c_{l_j}.
\end{align*}
We here used Newton's binomial series in the Dirichlet convolution algebra
and the fact that $1 = \mu^{*(-1)}$.
\end{proof}


\begin{thebibliography}{10}

\bibitem{C&M66}
L. Carlitz and L. Moser,
\newblock On some special factorizations of $(1-x{^n})/(1-x)$.
\newblock {\em Canad. Math. Bull.} {\bf 9} (1966), 421--426.

\bibitem{Comtet}
L. Comtet,
\newblock {\em Advanced Combinatorics}.
\newblock D. Reidel, Dordrecht, 1974.

\bibitem{Duncan2009}
B. Duncan and R. Peele,
\newblock Bell and stirling numbers for graphs.
\newblock {\em J. Integer Seq.} {\bf 12} (2009), Article 09.7.1.

\bibitem{GKP}
R. Graham, D. Knuth, and O. Patashnik,
\newblock {\em Concrete Mathematics},
\newblock Addison-Wesley, Boston, 2nd edition, 1994.

\bibitem{HHLS}
S. L. Hill, M. N. Huxley, M. C. Lettington, and K. M. Schmidt,
\newblock Some properties and applications of non-trivial divisor functions.
\newblock {\em Ramanujan J.} {\bf 51} (2020), 611--628.

\bibitem{2017SASS}
M. N. Huxley, M. C. Lettington, and K. M. Schmidt,
\newblock On the structure of additive systems of integers.
\newblock {\em Period. Math. Hung.} {\bf 78} (2019), 178--199.

\bibitem{L2025MF&DF}
Ambrose~D. Law, Matthew~C. Lettington, and Karl~Michael Schmidt,
\newblock Multifactorisations and divisor functions.
\newblock {\em Util. Math.} {\bf 125} (2025), 43--60.

\bibitem{LSDF&NSS}
M. C. Lettington and K. M. Schmidt,
\newblock Divisor functions and the number of sum systems.
\newblock {\em Integers} {\bf 20} (2020), \#A61.

\bibitem{LSfut}
M. C. Lettington and K. M. Schmidt,
\newblock On divisor function polynomials and {D}irichlet series,
\newblock {\em (in preparation)}, (2026).


\bibitem{MacMah}
P. A. MacMahon.
\newblock Second memoir on the composition of numbers.
\newblock {\em Philos. Trans. R. Soc., Ser. A}, {\bf 207} (1908), 65--134.

\bibitem{Munagi}
A. O. Munagi,
\newblock $k$-complementing subsets of nonnegative integers.
\newblock {\em Int. J. Math. Math. Sci.}
{\bf 2} (2005), 215--224.

\end{thebibliography}
\end{document}